\documentclass{llncs}
\usepackage{amssymb,amsmath}
\usepackage{lineno}
\usepackage{nicefrac}
\usepackage{url}
\usepackage{color}
\usepackage[american]{babel}
\usepackage{graphicx}
\usepackage{version}
\usepackage{subfigure} % change to subcaption if you'd rather be fancy
\usepackage[textsize=footnotesize]{todonotes}

\newif\iffull
\fulltrue
\graphicspath{{figures/}}

\usepackage{amsthm}
\newcommand{\leaveout}[1]{}

\date{}
\title{\hspace*{-5mm}A note on 1-planar graphs with minimum degree 7}
\author{Therese Biedl
\thanks{Supported by NSERC.}
}
\institute{David R.~Cheriton School of Computer
Science, University of Waterloo, Waterloo, Ontario N2L 1A2, Canada.
\email{biedl@uwaterloo.ca}
}
\begin{document}

\maketitle
\begin{abstract}
It is well-known that 1-planar graphs have minimum degree at most 7,
and not hard to see that some 1-planar graphs have minimum degree
exactly 7.  In this note we show that any such 1-planar graph has
at least 24 vertices, and this is tight.
\end{abstract}

%%%%%%%%%%%%%%%%%%%%%%%%%%%%%%%%%%%%%%%%%%%%%%%%%%%%%%%%%%%%%%%%%%%%%%%%
\section{Introduction}

A {\em 1-planar} graph is a graph that can be drawn in the plane
such that every edge has at most one crossing.  (We refer the reader
to any graph theory textbook for standard notations around graphs.)
This graph class was introduced by Ringel \cite{Ringel1965}, because
they naturally arise when taking a planar graph $G$ and its dual graph
$G^*$ and connecting vertices/faces that are incident to each other.
The resulting 1-planar graph $H$ has a vertex of degree at most 6, because at
least one of $G$ and $G^*$ has a vertex $v$ of degree 3, and adding
the edges to obtain $H$ doubles the degrees.  

However, not all 1-planar graphs are obtained by combining a planar
graph and its dual; for example $K_6$ is 1-planar but not obtained
in this way.  A different way to argue this is to exhibit a 1-planar
graph that has no vertex of degree 6 or less.  Figure~\ref{fig:ex}
shows one such graph; it has 24 vertices.  
In this note, we argue that ``24'' is tight.    Many other structural
properties are known for 1-planar graphs (see for example \cite{BEG+13,FM07}),
and even some for 1-planar graphs of minimum degree 7 \cite{HM11},
but it seems that the size of
1-planar graphs with minimum degree 7 has not previously been studied.

%%%%%%%%%%%%%%%%%%%%%%%%%%%%%%%%%%%%%%%%%%%%%%%%%%%%%%%%%%%%%%%%%%%%%%%%
\section{Lower bound}

\begin{figure}[ht]
\hspace*{\fill}
%\subfigure[~]
{\includegraphics[width=0.5\linewidth,page=2]{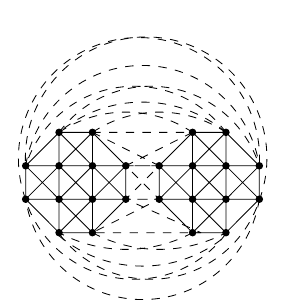}}
\hspace*{\fill}
\caption{A 1-planar graph with minimum degree 7.  
%Not all edges
%are shown; one should add any edge $(v_i,w_j)$ where $j\in \{i-1,i,i+1\}$
%(addition modulo 8).}
}
\label{fig:ex}
\end{figure}

Let $G$ be a simple 
1-planar graph with $n$ vertices, $m$ edges and minimum degree 7.
We assume throughout the following argument that one particular 1-planar
drawing $\Gamma$ (described by giving the clockwise order around each vertex and
the information which pairs of edges are crossing) has been fixed.  
We assume
that this drawing is {\em good} \cite{SchaeferBook}, which in particular means
that no edge crosses itself and no two edges with a common endpoint cross.

The given 1-planar drawing
defines the {\em regions}, which are the connected parts of $\mathbb{R}\setminus \Gamma$.
Each region $R$ contains on its boundary a number of vertices and crossings of $\Gamma$;
these are called the {\em corners} of $R$.  The {\em degree} of a region is the number
of corners that it has.  No region can have one or two corners, because
this would either mean a loop or a double edge in the graph, or the
drawing would not have been good.  

We call a 1-planar drawing {\em triangulated} if every region has
exactly three corners.  Note that the given 1-planar drawing $\Gamma$
can be made triangulated by inserting further edges as follows.
We already know that no region has fewer than three corners.
If some region $R$ has four or more corners, then we can find two
non-consecutive corners $v,w$ on $R$ that are both vertices, and 
add an edge between them.  Edge $(v,w)$ can be drawn inside $R$ without 
crossing, and either splits $R$ into two regions of smaller degree, or 
decreases the number of connected components of the boundary of $R$; 
either way we have made progress and can repeat until the resulting
drawing is triangulated and 1-planar.    
Note that this triangulated 1-planar drawing may not
represent a simple graph, but in what follows  simplicity will never
be used, only that the drawing is triangulated.

So fix from now on a triangulated 1-planar drawing $\Gamma$ where
all vertices have degree 7 or more.  Let $n$ and $m$ be the number
of edges in the graph that $\Gamma$ represents; note that $n\geq 3$
since the drawing is good and triangulated.
Let $n_7$ be the
number of vertices that have degree exactly 7.  Since all other vertices
have degree at least 8, we have
\begin{equation}
2m = \sum_{v\in V} \deg(v) \geq 7n_7 + 8(n-n_7) = 8n-n_7.
\end{equation}
It is well-known that every 1-planar simple graph with $n\geq 3$ vertices
has at most
$4n-8$ edges (see e.g.~\cite{FM07}), but we repeat the argument here because
our graphs are not necessarily simple (but have a triangulated drawing)
and because 
we need a more detailed bound.  Let $x$ be the number of crossings
of $\Gamma$.  Let $\Gamma^-$ be the planar drawing 
obtained by picking one edge at
each of these crossings and removing it.  Since $\Gamma$ was triangulated,
every crossing was surrounded by four regions that are triangles, and
removing one of the crossing edges hence gives two regions that are
triangles.  Therefore $\Gamma^-$ is a planar triangulated drawing.  It
may have multiple edges, but since all regions are triangles it nevertheless
follows from Euler's formula that $\Gamma^-$ has $3n-6$ edges.  Therefore
\begin{equation}
m = 3n-6+x.
\end{equation}
Let $t$ be the number of regions of $\Gamma$ that are {\em uncrossed}, i.e.,
for which none of the edges are crossed.  Put differently, $t$ counts
the number of regions where all three corners are vertices.  Returning
to the planar triangulated drawing $\Gamma^-$, observe that these $t$
regions of $\Gamma$ are also regions of $\Gamma^-$, and $\Gamma^-$ has
a further $2x$ regions that correspond to
the $x$ crossings of $\Gamma$.  This covers all regions of $\Gamma^-$.  By
Euler's formula the triangulated planar drawing $\Gamma^-$ has $2n-4$ regions,
and  therefore
\begin{equation}
2n-4 = 2x + t. 
\end{equation}
(In particularly therefore $m\leq 4n-8-\frac{t}{2}$ for any 1-planar graph with
a triangulated drawing.)

Consider a vertex $v$ of degree 7 in $\Gamma$.  Then no two edges $e_1,e_2$
that are incident to $v$ and consecutive in
the circular clockwise order at $v$ can be crossed, else the
region incident to $v$ and between $e_1$ and $e_2$ would not have
degree 3.  Since the degree of $v$ is odd, therefore there are two
consecutive uncrossed edges at $v$, say they 
are $(v,w_0)$ and $(v,w_1)$.  Then $\{v,w_0,v_1\}$ is an uncrossed
region since $\Gamma$ is triangulated.  So for every vertex of degree 7,
there exists an incident uncrossed region.  Vice versa, every
uncrossed region is incident to at most three vertices of degree 7,
and so we have
\begin{equation}
n_7\leq 3t.
\end{equation}
Now we put these equations together as follows:
$$
3n_7 
\stackrel{(1)}{\geq} 24n-6m
\stackrel{(2)}{\geq} 24n - (18n - 36 +6x)
\stackrel{(3)}{\geq} 6n+36 - (6n-12-3t) 
\stackrel{(4)}{\geq} 48 + n_7
$$
and the following holds.

\begin{theorem}
Any simple 1-planar graph with minimum degree 7 
has at least 24 vertices of degree 7.
\end{theorem}

In particular, $G$ cannot have fewer than 24 vertices overall.
Figure~\ref{fig:ex} shows that this is tight.  We also note
that the result holds more broadly for any 1-planar graph with
minimum degree 7 that has a 1-planar drawing where all regions
have degree three or more, even if the graph has multiedges or
loops.

%%%%%%%%%%%%%%%%%%%%%%%%%%%%%%%%%%%%%%%%%%%%%%%%%%%%%%%%%%%%%%%%%%%%%%%%
\section{Conclusion}

In this note, we argued that any simple 1-planar graph with minimum degree 7
has at least 24 vertices.  This bound is tight.
Our bound has some consequences for matching-bounds.  Wittnebel \cite{BW19}
showed that there is an infinite class of simple
1-planar graphs of minimum degree 7 
for which any matching $M$ has size at most $\frac{15n+16}{31}$.  He used
as ingredient a 1-planar graph with minimum degree 7 and 32 vertices.
Because the graph in Figure~\ref{fig:ex} has fewer vertices, this bound
can be improved.

\begin{lemma}
\label{lem:match}
For any $N$,
there exists a simple
1-planar graph with minimum degree 7 and $n\geq N$ vertices for which any matching
has size at most $\frac{11n+12}{23}$.
\end{lemma}
\begin{proof}
Let $n\geq N$ be such that $n=1\bmod 23$.  Create one vertex $v$, and split the
remaining vertices into $(n-1)/23$ groups of 23 vertices each.  For each group, insert
edges to turn these 23 vertices, plus vertex $v$, into the graph of Figure~\ref{fig:ex}.
Put differently, combine $(n-1)/23$ copies of this graph at one of their vertices.
Clearly the resulting graph is 1-planar, simple, and has minimum degree 7.  Observe that $G\setminus v$
has $(n-1)/23$ connected components, each of odd size.  The bound now follows from the ``easy''
direction of the Tutte-Berge theorem \cite{Berge1958}:  In any matching $M$, at most
one edge of $M$ connects $v$ to one component, so there are at least $\frac{n-1}{23}-1
= \frac{n-24}{23}$ vertices (one in each of the other components) that are unmatched.
Therefore any matching has size at most $\frac{1}{2}(n-\frac{n-24}{23}) = \frac{11n+12}{23}$.
\end{proof}

Because there are no smaller 1-planar graphs with minimum degree 7, there is
some hope for showing that the bound in Lemma~\ref{lem:match} is tight (i.e.,
there is a matching of size at least $\frac{11n+12}{23}$
in any 1-planar graph with minimum degree 7), though this remains an open problem.

%%%%%%%%%%%%%%%%%%%%%%%%%%%%%%%%%%%%%%%%%%%%%%%%%%%%%%%%%%%%%%%%%%%%%%%%
\bibliographystyle{plain}
\bibliography{journal,full,gd,papers}

\end{document}